\documentclass[11pt]{amsart}

\usepackage[all,cmtip]{xy}
\usepackage{amsmath, amssymb}
\usepackage{mathtools}
\usepackage[T1]{fontenc}
\usepackage[english]{babel}
\usepackage[hidelinks]{hyperref}
\usepackage[marginpar=2.5cm]{geometry}
\usepackage{xcolor}\usepackage{marginnote}
\usepackage{amsrefs}

\newtheorem{thm}{Theorem}[section]
\newtheorem{main}{Theorem}

\newtheorem*{thm*}{Theorem}
\newtheorem{lem}[thm]{Lemma}

\newtheorem{prop}[thm]{Proposition}
\newtheorem*{prop*}{Proposition}
\newtheorem{conj}[thm]{Conjecture}
\newtheorem{cor}[thm]{Corollary}

\theoremstyle{definition}
\newtheorem{defn}[thm]{Definition}
\newtheorem{notation}[thm]{Notation}
\newtheorem{remark}[thm]{Remark}
\newtheorem{question}[thm]{Question}

\newtheorem{example}[thm]{Example}

\def\e{\epsilon}
\def\la{\lambda}
\def\La{\Lambda}
\def\al{\alpha}
\def\bb{\mathbb}

\def\vp{\varphi}
\def\de{\delta}

\def\bb{\mathbb}

\def\G{\Gamma}
\def\cc{\mathcal}

\DeclareMathOperator{\id}{id}

\DeclareMathOperator{\CB}{CB}
\DeclareMathOperator{\cp}{{CP}}

\DeclareMathOperator{\cb}{cb}
\DeclareMathOperator{\Span}{span}

\DeclareMathOperator{\tr}{tr}

\DeclareMathOperator{\Ind}{Ind}
\DeclareMathOperator{\Ch}{Ch}

\DeclareMathOperator{\ucp}{UCP}
\DeclareMathOperator{\MIN}{MIN}

\newcommand{\cbnorm}[1]{\left\lVert #1 \right\rVert_{\text{cb}}}

\newcommand{\norm}[1]{\left\lVert #1 \right\rVert}
\newcommand{\bidual}[1]{#1^{**}}
\DeclareMathOperator{\coind}{Co-Ind}

\DeclareMathOperator{\UCP}{UCP}

\newcommand\ip[2]{\left\langle #1\, , #2 \right\rangle}

\DeclareMathOperator{\CP}{\text{CP}}

\begin{document}

\title{An Index for Inclusions of Operator Systems}

\author[Araiza]{Roy Araiza}
\author[Griffin]{Colton Griffin}
\author[Sinclair]{Thomas Sinclair}

\address{Department of Mathematics \& IQUIST, University of Illinois at Urbana-Champaign, Urbana, IL 61801}
\email{raraiza@illinois.edu}
\urladdr{https://math.illinois.edu/directory/profile/raraiza}

\address{Mathematics Department, Purdue University, 150 N. University Street, West Lafayette, IN 47907-2067}
\email{griff254@purdue.edu}
\email{tsincla@purdue.edu}
\urladdr{http://www.math.purdue.edu/~tsincla/}



\setcounter{tocdepth}{1}
\maketitle

\begin{abstract}
    Inspired by a well-known characterization of the index of an inclusion of II$_1$ factors due to Pimsner and Popa, we define an index-type invariant for inclusions of operator systems. We compute examples of this invariant, show that it is multiplicative under minimal tensor products, and explain how it generalizes the Lov\'asz theta invariant to general matricial systems in a manner that is closely related to the quantum Lov\'asz thetha invariant defined by Duan, Severini, and Winter.
\end{abstract}

\section{Introduction}

Jones \cite{Jones1983} defined a fundamental invariant for an inclusion of II$_1$ factors $\cc N\subset \cc M$, known as the \emph{index}, as the von Neumann dimension (coupling constant) of $L^2(\cc M)$ viewed as a left $\cc N$-module. Pimsner and Popa \cite{PimsnerPopa1986} gave a remarkable probabilistic formula for the index in terms of the smallest value $\la$ (if it exists) so that $\la\cdot E_{\cc N} - \id_{\cc M}:\cc M\to \cc M$ is positive, where $E_{\cc N}$ is the unique trace-preserving conditional expectation from $\cc M$ onto $\cc N$. In extending their work to the non-factorial case, it was realized by Baillet, Denizeau, and Havet \cite{Baillet1988} that a stronger condition than finiteness of the probabilistic index was necessary, namely the existence of $\la'$ so that $\la'\cdot E_{\cc N} - \id_{\cc M}$ is completely positive, though, as observed by Popa \cite[Remark 1.1.7]{PopaCBMS}, these notions coincide in the factorial case. We refer the reader to \cite{FrankKirchberg1998} for an account of these developments.

Let $\cc X_0\subset \cc X$ be an inclusion of operator systems. Following Pimsner and Popa, one can seek to determine how ``flat'' a unital, completely positive map $\vp: \cc X\to \cc X_0$ can possibly be. That is, one would like to know the smallest constant $\la$ so that $\la\cdot\vp - \id_{\cc X}: \cc X\to \cc X$ is (completely) positive for some unital, completely positive map $\vp: \cc X\to \cc X_0$. Heuristically at least, this is a reasonable generalization of the probabilistic index, since for an inclusion of $\cc N\subset \cc M$ of II$_1$ factors, the conditional expectation $E_{\cc N}$ commutes with all symmetries (i.e., $\ast$-automorphisms) of $\cc M$ which leave $\cc N$ fixed; hence, one would expect that $E_{\cc N}$ is the most uniform unital, complete positive map from $\cc M$ into $\cc N$ and that $\la$ is minimized for $E_{\cc N}$. As we will be working in the operator system category, it is natural to study the behavior of this invariant under injective tensor products, as they are functorial under inclusions, so the importance of $\la\cdot\vp - \id$ being completely positive becomes manifest. We term this the \emph{CP-index} of the inclusion $\cc X_0\subset \cc X$, which we denote by $\Ind_{\cp}(\cc X: \cc X_0)$. (The reader may go to Section \ref{sec:cp-index} below for the precise definition.)

Our main focus in this paper is the study of the $\cp$-index for inclusions of finite-dimensional operator systems, where it is always well-defined. Despite the generality of the definition, it turns out the $\cp$-index has many nice properties and is computable for many special cases. In the finite-dimensional case, we may reduce calculating the $\cp$-index to a problem in convex programming which allows use to deduce our first main result via strong duality:

\begin{main}[Corollary \ref{cor:multiplicativity-cp-index}]
     Let $\mathcal X_0 \subset \mathcal X$ and $\mathcal Y_0 \subset \mathcal Y$ be inclusions of finite-dimensional operator systems. Then \begin{align*}
     \Ind_{\cp}(\mathcal X \otimes_{\min} \mathcal Y: \mathcal X_0 \otimes_{\min} \mathcal Y_0) = \Ind_{\cp}(\mathcal X: \mathcal X_0) \Ind_{\cp}(\mathcal Y: \mathcal Y_0).
\end{align*}
\end{main}

This is an interesting property for several reasons. Among them, for a finite-dimensional system $\cc X$, this allows us to interpret $\tilde{\la}(\cc X) := \Ind_{\cp}(\cc X: \bb C1)$ as a ``noncommutative dimension'' for $\cc X$. Following Pimsner and Popa \cite{PimsnerPopa1986} and recent work of Gao, Junge, and LaRacuente \cite{Li2020} we may interpret $\log(\Ind_{\cp}(\cc X: \cc X_0))$ as a relative entropy theory, by which the above result establishes additivity under taking products. In very recent work relating relative von Neumann entropy of an inclusion of II$_\infty$ factors to the logarithm of the Jones index, Longo and Witten \cite[Section 7]{LongoWitten2022}, motivated by physical considerations, propose a notion of an index relative to a subspace rather than a subalgebra. It is hoped that the notion of index proposed here may likewise find applications to quantum field theory.

While our investigations started from the premise of generalizing the Pimsner--Popa index to operator systems, in the matricial case there are surprising connections between this notion of index and the quantum Lov\'asz theta invariant $\widetilde\vartheta$ defined by Duan, Severini, and Winter \cite{Winter2013}. In Section \ref{sec:examples} below, we give a semidefinite programming characterization of the $\cp$-index of a matricial system $\cc S\subset M_n$ from which we obtain the following result which shows that the $\cp$-index, like the quantum Lov\'asz theta invariant, is an extension of the Lov\'asz theta invariant to matricial systems.
\begin{main}[Proposition \ref{prop:cp-Lov\'asz-equal}]
    Let $\G = (V,E)$ be a graph on $n$-vertices. For its associated matricial system $\cc S_\G\subset M_n$ we have that 
    \[\Ind_{\cp}(M_n:\cc S_\G) = \vartheta(\G)\ \textup{and}\ \tilde\la(\bb CI_n + \cc S_\G^\perp) = \vartheta(\overline\G).\]
    Here $\overline\G$ denotes the graph complement of $\G$.
\end{main}
We leave it as an open question whether $\Ind_{\cp}(M_n:\cc S) = \widetilde\vartheta(\cc S)$ for general matricial systems. In that section we also compute other examples of the index for canonical inclusions of operator systems which further explore the analogy between the index and the Lov\'asz theta invariant \cite{Lovasz1979} and its quantum variants \cite{Boreland2021, Winter2013, Ganesan, PaulsenOrtiz2015}. These investigations are rather preliminary, but they should hopefully convince the reader that many deeper connections could be found in this direction.

In the final section, we define and investigate a couple of other index-type invariants. Based on the duality of finite-dimensional operator systems, we define a ``co-index'' for a pairing of an operator system $\cc X$ and a kernel $\cc J\subset \cc X$. The co-index seems related to work of Ortiz and Paulsen \cite{PaulsenOrtiz2015} on the extension of the (quantum) Lov\'asz theta invariant to operator systems, but is possibly distinct from it. There is again a relation between Lov\'asz theta invariant and the co-index in the matricial case, though we cannot determine whether they coincide. We conclude the paper by pointing out that the index could be generalized even further to inclusions of operator spaces or even Banach spaces. We leave it as the subject of future work whether more connections lie with Banach space geometry.

\section{Preliminaries}
We begin by reviewing the relevant background in operator space and operator system theory. Throughout the manuscript, $M_n$ will denote the algebra of complex $n \times n$ matrices. Given a Hilbert space $H$, a \emph{concrete operator space} is simply a closed linear subspace $\cc X \subset B(H)$. Of particular interest is the abstract formulation of operator spaces which was first studied in \cite{Ruan1978}. 

Consider a linear (vector) space $\cc X$. A norm $\al_n: M_n(\cc X) \to [0,\infty)$ on the linear space $M_n(\cc X)$ will be called a \emph{matrix norm}. 
\begin{defn}
    An \emph{abstract operator space} is defined to be the pair $(\cc X, \al)$ where $\cc X$ is a linear space, and $\al:= \{\al_n\}_{n \in \bb N}$ is a collection of matrix norms such that $(\cc X, \al_1)$ is a Banach space, and such that $\al$ satisfies the following two axioms: \begin{enumerate}
        \item For each $x \in M_n(\cc X), y \in M_m(\cc X)$ it follows $\al_{m+n}(x \oplus y) = \max \{\al_n(x), \al_m(y)\}$.
        \item For each $x \in M_n(\cc X)$ and scalar matrices $a,b \in M_n$ one has $\al_n(axb) \leq \norm{a} \al_n(x) \norm{b},$ where $\norm{\cdot}$ denotes the operator norm on $M_n.$ 
    \end{enumerate}
\end{defn} Given a linear space $\cc X$, then the sequence of matrix norms $\al:= \{\al_n\}_{n \in \bb N}, \al_n: M_n(\cc X) \to [0,\infty),$ is called an \emph{operator space norm} if the pair $(\cc X, \al)$ is an abstract operator space.

Given a linear map $u: \cc X \to \cc Y$ between two operator spaces $\cc X$ and $\cc Y$ then we define the \emph{n\textsuperscript{th}-amplification} $u_n: M_n(\cc X) \to M_n(\cc Y)$ as the map defined by \begin{align*}
    u_n(x) = \sum_{ij} e_{ij} \otimes u(x_{ij}),
\end{align*} where $x \in M_n(\cc X),$ and $e_{ij} \in M_n$ denotes the standard matrix units. We define the \emph{completely bounded (cb)} norm of a linear map $u: \cc X \to \cc Y$ as $\cbnorm{u: \cc X \to \cc Y}:= \sup_n \al_n(u_n).$ The map $u: \cc X \to \cc Y$ is \emph{completely bounded} if $\cbnorm{u} < \infty,$ a \emph{complete isometry} if $u_n$ is an isometry for each $n \in \bb N$, and a \emph{complete isomorphism} if it is a linear isomorphism such that $\cbnorm{u}, \cbnorm{u^{-1}} < \infty.$

Let $\cc X$ and $\cc Y$ be operator spaces and consider a linear map $u: \cc X \to \cc Y$. If $\cc N(u)$ denotes the null space of $u$, then consider the linear space $\cc X/ \cc N(u)$. We may equip this quotient space with an operator space structure in the following way: let $\pi: \cc X \to \cc X/ \cc N(u)$ denote the natural quotient map. Then given $\hat{x} \in M_n(\cc X/ \cc N(u))$, we define the corresponding norm as \begin{align*}
    \norm{\hat{x}}_{M_n(\cc X/ \cc N(u))}:= \inf \{ \norm{x}: x \in M_n(\cc X), \pi(x) = \hat{x}\}.
\end{align*} In particular, we define the matrix norm on $M_n(\cc X/ \cc N(u))$ by making the isometric identification $M_n(\cc X/ \cc N(u)) \simeq M_n(\cc X)/ M_n(\cc N(u)).$ The linear space $\cc X/ \cc N(u)$ with this structure is an operator space which we call the \emph{quotient operator space structure} on $\cc X/ \cc N(u)$. Of course, if $\cc J \subset \cc X$ is a closed subspace of the operator space $\cc X$, then one may consider such an operator space structure on $\cc X/ \cc J.$ Though, we will only be concerned with quotients of operator spaces by null spaces of particular maps. Thus, whenever we refer to quotients of operator spaces they will always be assumed to have this structure.
If $u: \cc X \to \cc Y$ is a linear map between operator spaces, then we say $u$ is a \emph{complete quotient map} if $\tilde{u}: \cc X/ \cc N(u) \to \cc Y$ is a complete isometry. We identify two operator spaces $\cc X $ and $\cc Y$ if there exists a linear map $u: \cc X \to \cc Y$ such that $u$ is a completely isometric isomorphism. If such a map exists between $\cc X$ and $\cc Y$ then we will denote this by $\cc X \simeq \cc Y.$ 

Given two operator spaces $\cc X$ and $\cc Y$, then the linear space of all completely bounded maps from $\cc X$ to $\cc Y$ is denoted $\CB(\cc X,\cc Y)$. If $\cc X = \cc Y$ then we denote the linear space as $\CB(\cc X)$. The space $\CB(\cc X, \cc Y)$ is complete with respect to the cb-norm and by making the identification $M_n(\CB(\cc X, \cc Y)) \simeq \CB(\cc X, M_n(\cc Y))$ then the linear space of completely bounded maps between operator spaces becomes an operator space. In Ruan's seminal work he proved that there is a one-to-one correspondence between abstract operator spaces and concrete operator spaces. We refer the interested reader to \cite{Ruan1978} for details. 

When considering operator spaces we will at times be interested in the \emph{dual operator space structure}. In particular, suppose $\cc X$ is an operator space and let $\cc X^*$ denote the Banach dual of $\cc X$. Then we may equip $\cc X^*$ with an operator space structure in the following way: given any $u \in M_n(\cc X^*)$ then we may identify $u$ with the map $U: \cc X \to M_n$ defined as $U(x):= \sum_{ij} e_{ij} \otimes u_{ij}(x)$, where $u = [u_{ij}]_{ij}.$ In particular, we define a matrix norm on $M_n(\cc X^*)$ via the identification $M_n(\cc X^*) \simeq \CB(\cc X, M_n).$ $\cc X^*$ with this structure will be called the \emph{operator space dual} of $\cc X$. 

Let $\cc X$ and $\cc Y$ be two operator spaces and let $\cc X \otimes \cc Y$ denote their algebraic tensor product. Analogous with the Banach space setting, one may consider various operator space structures on the linear space $\cc X \otimes \cc Y$, and the functoriality of various classes of linear maps with respect to said structures. Following \cite{blecher1991tensor} we have the following: 

\begin{defn}
    Let $\cc X$ and $\cc Y$ be two operator spaces. Then an \emph{operator space tensor product} of $\cc X$ and $\cc Y$ is an operator space structure $\al$ such that $(\cc X \otimes \cc Y, \al)$ is an operator space, which we denote as $\cc X \otimes_\al \cc Y$, and such that the following two conditions hold: \begin{enumerate}
        \item Given $x \in M_m(\cc X)$ and $y \in M_n(\cc Y)$ then \[\norm{x \otimes y}_{M_{mn}(\cc X \otimes_\al \cc Y)} \leq \norm{x}_{M_n(\cc X)} \norm{y}_{M_n(\cc Y)}.\]  
        \item Given $u: \cc X \to M_m$ and $v: \cc Y \to M_n$, both completely bounded, then $u \otimes v \in \CB(\cc X \otimes_\al \cc Y, M_{mn})$ and \begin{align*}
            \cbnorm{u \otimes v: \cc X \otimes_\al \cc Y \to M_{mn}} \leq \cbnorm{u} \cbnorm{v}.
        \end{align*}
    \end{enumerate}
\end{defn}

\begin{defn}
    We say that $\alpha$ is a \emph{reasonable operator space tensor norm} if for all operator spaces $\cc X, \cc Y, \cc S, \cc T$ and completely bounded maps $u: \cc X\to \cc S$ and $v: \cc Y\to \cc T$ we have that
    \[\cbnorm{u\otimes v: \cc X\otimes_\alpha \cc Y\to \cc S\otimes_\alpha\cc T}= \cbnorm{u}\cbnorm{v}.\]
\end{defn}

We now consider the relevant background for operator systems. If $H$ is a Hilbert space, then a \emph{concrete operator system} is a linear subspace $\cc X$ of $B(H)$, such that $\cc X$ contains the identity operator and is closed under the adjoint operation. Similar to the operator space case, of particular interest is the abstract formulation of operator systems which was developed in \cite{Choi1977}. Let $\cc X$ be a linear space with an involution $*: \cc X \to \cc X$. We will typically call such a space a \emph{$*$-vector space}. The (real) subspace of hermitian elements will be denoted $\cc X_h$. As operator spaces were completely determined by their matricial norm structure, operator systems are completely determined by their matricial order structure. We consider the pair $(\cc X, \cc X^+)$ where $\cc X$ is a $*$-vector space, and $\cc X^+ \subset \cc X_h$ is a \emph{(positive) cone} in $\cc X$, which is to say $\cc X^+$ satisfies the following properties: (1) $\cc X^+ + \cc X^+ \subset \cc X^+$; (2) $r \cc X^+ \subset \cc X^+$ for all $r \geq 0$; (3) $\cc X^+ \cap -\cc X^+ = \{0\}.$ Such a cone induces a partial ordering on $\cc X_h$ by declaring $x \leq y$ if and only if $y - x \in \cc X^+$. We call the pair $(\cc X, \cc X^+)$ an \emph{ordered $*$-vector space}. An \emph{Archimedean order unit} of $\cc X$ is defined to be an element $1 \in \cc X_h$ satisfying the following two properties: (1) for all $x \in \cc X_h$ there exists $r>0$ such that $r1 - x \in \cc X^+$; (2) if $\e 1 + x \in \cc X^+$ for all $\e >0$ then $x \in \cc X^+$. The existence of such an element $1$ ensures that the positive cone majorizes the hermitian subspace $\cc X_h$, and furthermore that the cone of positive elements of $\cc X$ is ``order'' closed. The triple $(\cc X, \cc X^+, 1)$ consisting of a $*$-vector space $\cc X$, a positive cone $\cc X^+$, and an Archimedean order unit $1$ will be called an \emph{Archimedean order unit (AOU) space}. We will typically denote an AOU space as $\cc X$ when no confusion will arise. 

A \emph{matrix ordering} on an ordered $*$-vector space $\cc X$ is a collection $\cc C:= \{\cc C_n\}_{n \in \bb N}$ satisfying the following conditions: \begin{enumerate}
    \item $\cc C_n \subset M_n(\cc X)_h$ is a cone for all $n \in \mathbb N$. 
    \item $\cc C_n \cap -\cc C_n = \{0\}$ for each $n \in \bb N$. 
    \item $a^* \cc C_n a \subset \cc C_m$ for all $a \in M_{n,m}, n,m \in \bb N$. 
\end{enumerate} If $\cc X$ is a $*$-vector space and $\cc C$ is a matrix ordering on $\cc X$, then we will call the pair $(\cc X, \cc C)$ a \emph{matrix ordered $*$-vector space}. Note that a $*$-vector space $\cc X$ together with a matrix ordering $\cc C$ necessarily yields an ordered $*$-vector space $(M_n(\cc X), \cc C_n)$ for each $n \in \bb N$, with the involution on $M_n(\cc X)$ defined in the natural way. We may at times alternate between the notation $\cc C_n$ and $M_n(\cc X)^+$ when discussing a matricial cone of a matrix ordered $*$-vector space $\cc X$. An element $1 \in \cc X_h$ is an \emph{Archimedean matrix order unit} if for each $n \in \mathbb N,$ $I_n \otimes 1$ is an Archimedean order unit for the ordered $*$-vector space $(M_n(\cc X), \cc C_n)$. 

\begin{defn}
    An \emph{abstract operator system} is a triple $(\cc X, \cc C, 1)$ where $(\cc X, \cc C)$ is a matrix ordered $*$-vector space, and $1$ is an Archimedean matrix order unit for $(\cc X, \cc C)$.
\end{defn} When no confusion will arise we will denote an abstract operator system as $\cc X$. Let $(\cc X, \cc C, 1_X)$ and $(\cc Y, \cc D, 1_Y)$ be two operator systems and let $u: \cc X \to \cc Y$ be a linear map. We say $u$ is \emph{positive} if $u(\cc X^+) \subset \cc Y^+$, \emph{completely positive} if $u_n(\cc C_n) \subset \cc D_n$ for each $n \in \mathbb N$, and \emph{unital} if $u(1_X) = 1_Y.$ If $u: \cc X \to \cc Y$ is a linear isomorphism such that both $u$ and $u^{-1}$ are completely positive, then we will say $u$ is a \emph{complete order isomorphism}. If there exists a complete order isomorphism between two operator systems, then we will say they are \emph{completely order isomorphic}. 

\begin{notation} The set of all completely positive maps between operator systems $\cc X$ and $\cc Y$ will be denoted $\CP(\cc X, \cc Y)$, and the set of all unital completely positive maps between $\cc X$ and $\cc Y$ will be denoted $\UCP(\cc X, \cc Y)$. If $\cc X = \cc Y$ then the respective completely positive and unital completely positive maps from $\cc X$ to itself will be denoted $\CP(\cc X)$ and $\UCP(\cc X)$. The set of all completely positive maps from an operator system $\cc X$ to itself satisfying $u(\bb C 1_X) = \bb C 1_X$ will be denoted $\CP_1(\cc X).$
\end{notation}

As Ruan did for operator spaces, a seminal result of Choi and Effros says that there is a one-to-one correspondence between abstract operator systems and concrete operator systems. In particular, given any abstract operator system $\cc X$ then there exists a Hilbert space $H$, a concrete operator system $\tilde{\cc X} \subset B(H)$, and a unital complete order isomorphism $u: \cc X \to \tilde{\cc X}.$ We refer the interested reader to \cite{Choi1977} for the details. 

Let $\cc X$ be an operator system. Then a \emph{kernel} in $\cc X$ is defined to be a non-unital $*$-closed subspace $\cc J \subset \cc X$ such that there exists an operator system $\cc Y$ and a (nonzero) completely positive map $u: \cc X \to \cc Y$ for which $\cc J = \cc N(u).$ Consider the $*$-vector space $\cc X/ \cc J$ with the involution defined in the natural way. If $\pi: \cc X \to \cc X/ \cc J$ denotes the natural quotient mapping then we define an operator system structure on the quotient space in the following way: consider the collection $\cc R:= \{\cc R_n\}_{n \in \bb N}$ where \begin{align*}
    \cc R_n:= \{&[x_{ij} + \cc J]_{ij} \in M_n(\cc X/\cc J): \forall \e >0 \,\, \exists y \in M_n(\cc J) \,\,\,\text{such that} \\
    &\e I_n \otimes 1 + (x+y) \in M_n(\cc X)^+\}.
\end{align*} One may then show that the collection $\cc R$ is a matrix ordering on $\cc X/ \cc J$ and $1 + \cc J$ is an Archimedean matrix order unit. The \emph{quotient operator system} $\cc X/ \cc J$ is defined to be the triple $(\cc X/ \cc J, \cc R, 1 + \cc J).$

Some remarks regarding duality in operator system theory are in order. Given an arbitrary operator system $\cc X$ then it is not true that its Banach dual $\cc X^*$ is an operator system in general. Though, the resulting dual space is always a matrix ordered $*$-vector space, and in particular, is an operator space. If one restricts to finite-dimensional operator systems, then it follows the Banach dual is always an operator system. We refer the reader to \cite[Corollary 4.5]{Choi1977} for the proof of this latter fact.

We now review the relevant tensor theory for operator systems which was developed in \cite{KPTT2011}. Let $(\cc X, \cc C, 1_X)$ and $(\cc Y, \cc D, 1_Y)$ be two operator systems. An \emph{operator system structure} on $\cc X \otimes \cc Y$ is a collection of cones $\al := \{\cc R_n\}_{n \in \bb N}, \cc R_n \subset M_n(\cc X \otimes \cc Y)_h$ satisfying the following properties: \begin{enumerate}
    \item $(\cc X \otimes \cc Y, \al, 1_X \otimes 1_Y)$ is an operator system, which we denote as $\cc X \otimes_\al \cc Y$. 
    \item $\cc C \otimes \cc D \subset \cc R$. This is to say $\cc C_n \otimes \cc D_m \subset \cc R_{mn}$, for $m,n \in \bb N.$
    \item If $u: \cc X \to M_n$ and $v: \cc Y \to M_m$ are two unital completely positive maps then $u \otimes v: \cc X \otimes_\al \cc Y \to M_{mn}$ is unital completely positive. 
\end{enumerate}

\begin{defn}
    An \emph{operator system tensor product} is a mapping $\al: \text{OpSys} \times \text{OpSys} \to \text{OpSys}$ such that given any two operator systems $\cc X, \cc Y$, then $\al(\cc X, \cc Y)$ is an operator system structure on $\cc X \otimes \cc Y$. The resulting operator system will be denoted $\cc X \otimes_{\al} \cc Y$. 
\end{defn}

An operator system tensor product is called \emph{functorial} if given any operator systems $\cc X, \cc Y, \cc S, \cc T$ and $u \in \UCP(\cc X, \cc S), v \in \UCP(\cc Y, \cc T)$, then $u \otimes v \in \UCP(\cc X \otimes_\al \cc Y, \cc S \otimes_\al \cc T).$

\section{The CP-index}\label{sec:cp-index}
For $\cc X_0\subset \cc X$ an inclusion of operator systems, we define the \emph{index} of $\cc X_0$ in $\cc X$ to be the value
    \begin{equation*}
        \Ind(\cc X : \cc X_0) := \inf\{\|\phi(1)\| : \phi\in \cp_1(\cc X),\ \phi(\cc X)\subset \cc X_0,\ \phi\succ\succ\id\}
    \end{equation*}
if it exists and $\Ind(\cc X: \cc X_0)=\infty$ otherwise. Here, we use $\phi\succ\succ \psi$ to denote that the map $\phi-\psi$ is positive, though not necessarily completely positive.

\begin{defn}
Given an inclusion of operator systems $\cc X_0 \subset \cc X$ we define the \emph{CP-index} to be  \begin{align} 
\Ind_\text{CP}(\cc X, \cc X_0):= \inf \{ \norm{\vp(1)}: \vp \in \CP_1(\cc X), \vp(\cc X) \subset \cc X_0, \vp - \id \in \CP(\cc X)\}
\end{align}
if it exists and $\Ind_{\cp}(\cc X: \cc X_0)=\infty$ otherwise.
\end{defn}

In order to help with notation we define the following sets: given an inclusion of operator systems $\mathcal X_0 \subset \mathcal X$ we define \begin{equation}
    \Lambda(\cc X, \cc X_0) := \{ \vp \in \text{CP}_1(\mathcal X): \vp(\mathcal X) \subset \mathcal X_0,\ \vp - \id \in \CP(\cc X)\}. 
\end{equation} Thus, the $\cp$-index becomes 
\begin{align*}
    \Ind_{\text{CP}}(\mathcal X, \cc X_0):= \inf \{\norm{\vp(1)} : \vp \in \Lambda(\cc X, \cc X_0)\}.
\end{align*} 

\begin{remark}
    Note that if $\cc X$ is finite-dimensional, then $\La(\cc X, \cc X_0)$ is non-empty. This can be inferred from Lemma \ref{lem:feasible-interior} below.
\end{remark}

\begin{prop}
    Suppose that $\cc X$ is a finite-dimensional operator system. For any proper subsystem $\cc X_0$ we have that $\Ind_{\cp}(\cc X: \cc X_0)> 1$. 
\end{prop}

\begin{proof}
    Indeed, we have that $\La(\cc X,\cc X_0)$ is closed in $\CB(\cc X)$, which is locally compact since $\cc X$ is finite dimensional. We have that $\La(\cc X, \cc X_0)$ is non-empty by the previous remark. Thus if $\Ind_{\cp}(\cc X, \cc X_0)=1$ there is a unital completely positive map $\vp: \cc X\to \cc X_0$ so that $\vp - \id\in \cp(\cc X)$. However, $(\vp - \id)(1)=0$, so this would imply $\vp = \id$.
\end{proof}

We begin by proving that the CP-index is submultiplicative relative to functorial injective operator system tensor products. 

\begin{lem}\label{lem: functorial op sys tensor functorial wrt CP_1}
Let $\mathcal X$ and $\mathcal Y$ be two operator systems, and suppose $\alpha$ is a functorial operator system tensor product. Then $\alpha$ is functorial relative to elements of $\cp_1$. That is, if $u \in \cp_1(\mathcal X)$ and $v \in \cp_1(\mathcal Y)$, then $u \otimes v\in \cp_1(\mathcal X \otimes_\alpha \mathcal Y).$ 
\end{lem}

\begin{proof}
We may assume $u$ and $v$ are non-zero, so we have $u(1_X) = s\, 1_X$ and $v(1_Y) = t\, 1_Y$ for some $s,t>0$. Thus, we have that $\frac{1}{s}\, u \in \text{UCP}(\mathcal X)$ and $\frac{1}{t}\, v \in \text{UCP}(\mathcal Y)$. By functoriality of $\alpha$ it then follows $\frac{1}{st}\, u\otimes v \in \ucp(\cc X\otimes_\alpha \cc Y)$. This proves the claim since the action by a positive scalar preserves complete positivity. \qedhere
\end{proof}

Given an inclusion of operator systems $\mathcal X_0 \subset \mathcal X, \mathcal Y_0 \subset \mathcal Y$, and an operator system tensor product $\alpha$, then we let $\mathcal X_0 \otimes_\text{rel} \mathcal Y_0$ denote the \emph{relative} operator system tensor product, i.e., the operator system structure on $\cc X_0\otimes \cc Y_0$ obtained by restriction of $\mathcal X \otimes_\alpha \mathcal Y$.

\begin{prop}\label{prop: CP-index is submultiplicative}
Let $\mathcal X_0 \subset \mathcal X$ and $\mathcal Y_0 \subset \mathcal Y$ be two inclusions of operator systems and let $\alpha$ be a functorial operator system tensor product. Set $\cc R:= \mathcal X \otimes_\alpha \mathcal Y$ and similarly set $\mathcal R_0 := \mathcal X_0 \otimes_\text{rel}\mathcal Y_0$. Then \begin{align*}
    \Ind_{\cp}(\mathcal R, \mathcal R_0) \leq \Ind_{\cp}(\mathcal X, \mathcal X_0) \Ind_{\cp}(\mathcal Y, \mathcal Y_0).
\end{align*}
\end{prop}

\begin{proof}
We first prove $\Lambda(\cc X,\cc X_0) \otimes \Lambda(\cc Y,\cc Y_0) \subset \Lambda(\cc R,\cc R_0).$ To this end, let $u\in \Lambda(\cc X,\cc X_0)$ and let $v \in \Lambda(\cc Y,\cc Y_0),$.  We have $u\otimes v(\cc R) \subset \mathcal R_0$ by linearity, and by Lemma~\ref{lem: functorial op sys tensor functorial wrt CP_1} we know $u\otimes v \in \CP_1(\mathcal R).$  It remains to show $u \otimes v - \id_{\cc R} \in \CP_1(\mathcal R). $ By assumption, $u - \id_{\mathcal X} \in \CP_1(\mathcal X)$; therefore, $(u - \id_{\mathcal X}) \otimes \id_{\mathcal Y} \in \CP_1(\mathcal R)$ by Lemma~\ref{lem: functorial op sys tensor functorial wrt CP_1}. Similarly, it follows that $u \otimes (v - \id_{\mathcal Y}) \in \CP_1(\mathcal R).$ Thus, \begin{align*}
    u\otimes v - \id_{\cc R} = u\otimes (v - \id_{\cc Y}) + (u - \id_{\cc X})\otimes \id_{\cc Y}\in \CP_1(\cc R).
\end{align*}
This proves that $u\otimes v \in \Lambda(\cc R,\cc R_0).$ We therefore deduce the following: 
\begin{align*}
    \Ind_{\cp}(\mathcal R: \mathcal R_0) &\leq \inf\{\|u \otimes v(1_{\mathcal R})\|: u\otimes v\in \Lambda(\cc X,\cc X_0) \otimes \Lambda(\cc Y,\cc Y_0)\} \\
        &=\inf \{\|u(1_{\mathcal X})\|\cdot \|v(1_{\mathcal Y})\|: u\in \Lambda(\cc X,\cc X_0), v\in \Lambda(\cc Y,\cc Y_0)\} \\
        &= \Ind_{\cp}(\mathcal X: \mathcal X_0) \Ind_{\cp} (\mathcal Y: \mathcal Y_0). \qedhere
\end{align*}
\end{proof}

As an immediate corollary we deduce the following: 

\begin{cor}\label{cor:min-submultiplicative}
Let $\mathcal X_0 \subset \mathcal X$ and $\mathcal Y_0 \subset \mathcal Y$ be inclusions of operator systems. Then \begin{align*}
     \Ind_{\cp}(\mathcal X \otimes_{\min} \mathcal Y: \mathcal X_0 \otimes_{\min} \mathcal Y_0) \leq \Ind_{\cp}(\mathcal X: \mathcal X_0) \Ind_{\cp}(\mathcal Y: \mathcal Y_0).
\end{align*}
\end{cor}

We now turn to proving multiplicativity of the $\cp$-index for inclusions of finite-dimensional operator systems. For the remainder of the section, $\cc X$ will be a finite-dimensional operator system. We are only really interested in $\cc X$ or its dual $\cc X^*$ as a Banach space with a canonical real (hermitian) structure. 

Let $\cc X_0\subset \cc X$ be a subsystem. One can always construct a hermitian basis $\{x_0=1, x_1,\dotsc, x_n\}$ for $\cc X$ with the property that $\cc X_0 = \Span\{x_0, \dotsc, x_k\}$ for some $1\leq k\leq n$. We fix $k$ throughout. Given such a basis, let $\{\de_0,\de_1,\dotsc, \de_n\}$ be a complimentary Auerbach basis which induces the hermitian structure on $\cc X^*$. The element $\id_{\cc X}: \cc X\to \cc X$ can be identified with the tensor $\sum_{i=0}^n \delta_i \otimes x_i$ which can be, in turn, viewed as ``a scale of the maximally entangled state'' on $\cc X^*\otimes \cc X$ as defined in \cite[Appendix A]{Kavruk2015}.

Consider $\cc C := \CP_1(\cc X)\subset \cc X^*\otimes \cc X$ which is a hermitian cone under the canonical hermitian structure on $\cc X^*\otimes \cc X$. 
We denote by $\cc C^\circ\subset \cc X\otimes \cc X^*$ the polar cone which is also hermitian. 

\begin{prop}
The following convex program computes $\Ind_{\cp}(\cc X: \cc X_0)-1$:

\begin{equation}\label{index-opsys-primal}
    \begin{aligned}
        & \textup{minimize} && \ip{\vp}{1\otimes \de_0} &&&\\
        & \textup{s.t} && \ip{\vp+\id_{\cc X}}{x_i\otimes \de_j} = 0, &&& 1\leq i\leq n; k < j \leq n\\
        & && \vp\in \cc C. &&&\\
    \end{aligned}
\end{equation}
\end{prop}

\begin{proof}
    Let us make the change of variables $\vp = \psi - \id_{\cc X}$. Since $\vp,\psi\in \cc C$ we have that $\|\vp(1)\| = \ip{\vp}{1\otimes\de_0} = \|\psi\| -1$. The first constraint shows that $\psi(\cc X)\subset \cc X_0$, while the second constraint shows that $\psi\in \cp_1(\cc X)$. \qedhere
\end{proof}

Notice that $\cc X\otimes \cc X_0^* = \Span\{x_i\otimes\de_j : i=1,\dotsc,n;\ j=k+1,\dotsc,n\}$. Therefore, following \cite[section 4.7]{GartnerMatousek2012}, the dual program to this is:

\begin{equation}\label{index-opsys-dual}
    \begin{aligned}
        & \text{maximize} && \ip{\id_{\cc X}}{f} &&&\\
        & \text{s.t.} && f\in \cc X\otimes (\cc X_0)^\perp &&&\\
        & && f + 1\otimes \de_0\in \cc C^\circ.
    \end{aligned}
\end{equation}

\begin{lem}\label{lem:feasible-interior}
    The feasible set of (\ref{index-opsys-primal}) has an interior point.
\end{lem}

\begin{proof}
    By \cite[Lemma 2.5]{Kavruk2014}, we have that $\de_0$ is an order unit for $\cc X^*$, so $\de_0\otimes 1$ is an order unit for $\cc X^*\otimes_{\min} \cc X$. By \cite[Proposition 6.1]{Kavruk2014}, we have that $\CP(\cc X)$ is identified with the positive cone in $\cc X^*\otimes_{\min} \cc X$; hence, there is a constant $K$ so that for all $L\geq K$, we have that $L\, \de_0\otimes 1 - \id\in \CP_1(X) = \cc C$.
\end{proof}

\begin{prop}\label{prop:super-multiplicative}
    The values of the convex programs (\ref{index-opsys-primal}) and (\ref{index-opsys-dual}) agree, and we have that for any two pairs of inclusions of finite-dimensional operator systems $\cc X_0\subset \cc X$ and $\cc Y_0\subset \cc Y$ that
    \begin{equation}
        \Ind_{\cp}(\cc X\otimes_{\min} \cc Y: \cc X_0\otimes_{\min} \cc Y_0)\geq \Ind_{\cp}(\cc X: \cc X_0)\cdot \Ind_{\cp}(\cc Y: \cc Y_0).
    \end{equation}
\end{prop}

\begin{proof}
    The arguments in \cite[sections 4.6 and 4.7]{GartnerMatousek2012} can be seen to apply equally well to hermitian cones in complex finite-dimensional vector spaces by taking the real subspace spanned by the cone. Hence, by Strong Duality for convex programs \cite[Theorem 4.7.3]{GartnerMatousek2012} and Lemma \ref{lem:feasible-interior} the convex program (\ref{index-opsys-primal}) and its dual (\ref{index-opsys-dual}) achieve the same value.
    
    Let $f\in \cc X\otimes (\cc X_0)^\perp$ and $g\in \cc Y\otimes (\cc Y_0)^\perp$ be optimal solutions to the program (\ref{index-opsys-dual}) for $\cc X_0\subset \cc X$ and $\cc Y_0\subset \cc Y$, respectively, and let \[h = f\otimes g + (1_X\otimes \de_0^X)\otimes g + f\otimes (1_Y\otimes \de_0^Y).\] It can be checked that $h\in (\cc X\otimes \cc Y)\otimes (\cc X_0\otimes\cc Y_0)^\perp$ and that
    \[h + (1_X\otimes 1_Y)\otimes (\de_0^X\otimes\de_0^Y) = (f + 1_X\otimes \de_0^X)\otimes (g + 1_Y\otimes\de_0^Y)\in \CP_1(\cc X\otimes_{\min}\cc Y)^\circ,\]
    since $\cp_1(\cc X)^\circ\otimes\cp_1(\cc Y)^\circ\subset \cp_1(\cc X\otimes \cc Y)^\circ$.
    Thus, by some basic arithmetic
    \begin{equation}
        \begin{split}
            \Ind_{\cp}(\cc X\otimes_{\min} \cc Y: \cc X_0\otimes_{\min} \cc Y_0) -1 &\geq \ip{\id}{h}\\
            &= \ip{\id_{\cc X}}{f}\ip{\id_{\cc Y}}{g} + \ip{\id_{\cc X}}{f} + \ip{\id_{\cc Y}}{g}\\
            &=\Ind_{\cp}(\cc X: \cc X_0)\cdot \Ind_{\cp}(\cc Y: \cc Y_0) -1.
        \end{split}
    \end{equation}
    The result then follows. \qedhere
\end{proof}

Combining Proposition \ref{prop:super-multiplicative} with Corollary \ref{cor:min-submultiplicative} we have

\begin{cor}\label{cor:multiplicativity-cp-index}
   Let $\mathcal X_0 \subset \mathcal X$ and $\mathcal Y_0 \subset \mathcal Y$ be inclusions of finite-dimensional operator systems. Then \begin{align*}
     \Ind_{\cp}(\mathcal X \otimes_{\min} \mathcal Y: \mathcal X_0 \otimes_{\min} \mathcal Y_0) = \Ind_{\cp}(\mathcal X: \mathcal X_0) \Ind_{\cp}(\mathcal Y: \mathcal Y_0).
\end{align*}
\end{cor}

The following special case of the previous result is worth pointing out in its own right.

\begin{cor}
    For any inclusion $\cc X_0\subset \cc X$ of finite-dimensional operator systems and all $n\in \bb N$ we have that 
    \[\Ind_{\cp}(M_n(\cc X) : M_n(\cc X_0)) = \Ind_{\cp}(\cc X: \cc X_0).\]
\end{cor}

For a finite-dimensional operator system $\cc X$, we see that
\begin{equation}
    \tilde{\la}(\cc X):=\Ind_{\cp}(\cc X: \bb C1)
\end{equation} can be thought of as a sort of ``quantum dimension'' for $\cc X$ since $\tilde{\la}(\cc X\otimes_{\min} \cc Y) = \tilde{\la}(\cc X)\tilde{\la}(\cc Y)$. To further strengthen the analogy, we can consider the direct sum $\cc X\oplus_\infty \cc Y$. It is straightforward to check that $\tilde{\la}(\cc X\oplus_\infty \cc Y)=\max\{\tilde{\la}(\cc X),\ \tilde{\la}(\cc Y)\}$. As an illustration, we compute this number for matrix algebras. 

\begin{prop}\label{prop:trace}
    We have that $\tilde{\la}(M_n) = n$ for all $n\in \bb N$.
\end{prop}

\begin{proof}
    We fix a completely positive map $\Phi: M_n\to \bb C1$ which is given by $\Phi: X\mapsto \phi(X)1$, where $\phi:M_n\to \bb C$ is a state. For any state $\phi$ we have that 
    \[\tau_n(x) = \int_{U(n)} \phi(uxu^*) du,\] where $\tau_n$ is the normalized trace on $M_n$ and $du$ is the Haar measure on the unitary group $U(n)$ of $M_n$ since the right hand side is unital and invariant under conjugation by $U(n)$. 
    
    Notice that if $\Psi: M_n\to M_n$ is any map so that $c\cdot\Psi- \id\in \cp(M_n)$ for some $c$ and $u\in U(n)$, setting
    \begin{equation}
        \Psi^u(x) := u\Psi(u^*xu)u^*
    \end{equation}
    we have that $c\cdot\Psi^u(x)- \id\in \cp(M_n)$ as well. Thus if $c\cdot\Phi- \id\in \cp(M_n)$ for all, it follows that $T_n :x \mapsto c\cdot \tau_n(x) 1- x$ is completely positive. The Choi matrix of $T_n$ is 
    \begin{equation*}
        \sum_{i,j} E_{ij}\otimes T_n(E_{ij}) = \sum_{i,j} \frac{c}{n}\cdot E_{ii}\otimes E_{jj} - E_{ij}\otimes E_{ij},
    \end{equation*}
    so it is easy to see that the Choi matrix is positive, hence $T_n$ is completely positive (see the proof of Proposition \ref{prop:cp-ind-primal} below), if and only if $c\geq n$. \qedhere
\end{proof}

\begin{remark}
    By Hahn--Banach it follows that $\tilde{\la}(\cc S)\leq n$ for all subsystems $\cc S\subset M_n$. It is an easy exercise to check that $\tilde{\la}(D_n)=n$ where $D_n\subset M_n$ is the diagonal matrices. In particular this implies that $D_n$ is not completely order isomorphic to a subsystem of $M_k$ for any $k< n$.  It follows that if $\cc S_\G\subset M_n$ is the subsystem given by a graph $\G = (V,E)$ with $|V|=n$ as defined at (\ref{eq:S-gamma}) below, then $\tilde{\la}(S_\G)=n$.
\end{remark}

We end this section with the following natural question. It is clear from the definitions that $\Ind(M_n(\cc X): M_n(\cc X_0))\leq \Ind_{\cp}(\cc X: \cc X_0)$ for all $n\in\bb N$.

\begin{question}
    Does $\Ind_{\cp}(\cc X: \cc X_0) = \limsup_{n\to\infty} \Ind(M_n(\cc X): M_n(\cc X_0))$ for any inclusion of operator systems?
\end{question}
We point out that in the matricial system case, the above question is immediate by \emph{Smith's Lemma} \cite[Lemma B.4]{BrownOzawa} which states that given any bounded map $\varphi: \cc X \to M_n$, where $\cc X$ is an operator space, then $\vp$ is completely bounded with $\cbnorm{\vp} = \norm{\vp\otimes \id_{M_n}}.$

\begin{remark}
    In the case of an inclusion of II$_1$ factors $\cc M_0\subset \cc M$, it is immediate that $\Ind_{\cp}(\cc M: \cc M_0)\leq [\cc M : \cc M_0]$, while equality remains unclear except in the case that $(\cc M_0\subset \cc M)\cong (M_k\otimes \cc N\subset M_n\otimes\cc N)$ for some inclusion $M_k\subset M_n$.
    
    For a finite-index inclusion of II$_1$ factors $\cc N\subset \cc M$, one can consider the operator system $\bb X_{\cc M}^{\cc N}$ consisting of the space of all completely bounded (normal) $\cc N$-bimodular maps from $\cc M$ to itself with $\id_{\cc M}$ serving as the order unit. That $\bb X_{\cc M}^{\cc N}$ is finite-dimensional follows from \cite{PimsnerPopa1986}. For the inclusion $1\otimes\cc N\subset M_n\otimes\cc N$, we have that $\bb X_{\cc M}^{\cc N}$ is canonically isomorphic to $M_n\otimes M_n$; hence,
    \begin{conj}
        $\tilde\la(\bb X_{\cc M}^{\cc N}) = [\cc M:\cc N]^2.$
    \end{conj}
\end{remark}

\section{Examples and applications}\label{sec:examples}

\subsection{The $\cp$-index of a matricial system}\label{sec:cp-matricial}
We begin by showing that for matricial systems $\cc S\subset M_n$ the $\cp$-index $\Ind_{\cp}(M_n:\cc S)$ is computable by a semidefinite program. We then identify the dual program and use this to relate $\Ind_{\cp}(M_n:\cc S)$ to the quantum Lov\'asz theta invariant $\widetilde\vartheta(\cc S)$ defined by Duan, Severini, and Winter \cite{Winter2013}.  For background material on complex semidefinite programming, we refer the reader to \cite[Section 2]{AGKS}.

\begin{prop}\label{prop:cp-ind-primal}
    Let $\cc S\subset M_n$ be a matricial system. We have that $\la = \Ind_{\textup{CP}}(M_n:\cc S)^{-1}$ can be expressed by the following semidefinite program:
    \begin{equation}\label{eq:index-primal}
        \begin{aligned}
            &\textup{maximize} &&\la \\
            &\textup{subject to} &&\tr\otimes\id(X) = (1-\la)I_n\\
            & && X +\la\Delta_n\in M_n\otimes\cc S\\
            & && X\in (M_n\otimes M_n)^+.
        \end{aligned}
    \end{equation}
    Here $\Delta_n = \sum_{i,j} E_{ij}\otimes E_{ij}$.
\end{prop}

\begin{proof}
    Indeed, given a linear map $\Phi: M_n\to M_n$, we define the \emph{Choi matrix} of $\Phi$
    \begin{equation}\label{eq:choi-matrix}
        \Ch(\Phi) := \sum_{i,j} E_{ij} \otimes \Phi(E_{ij}) \in M_n\otimes M_n.
    \end{equation}
    It is well-known that $\Phi$ is completely positive if and only if $\Ch(\Phi)$ is positive semidefinite. (See, for instance, \cite[Theorem 3.14]{paulsen2002completely}.) Note that $\Delta_n$ is the Choi matrix of the identity map. To see that the program (\ref{eq:index-primal}) computes the $\cp$-index of $\cc S$ in $M_n$, we make a change of variables $X = Y - \la \Delta_n = \Ch(\Phi - \la\id)$. The last line thus reads, ``$\Phi - \la\id$ is completely positive,'' while the second-to-last line is equivalent to $\Phi(M_n)\subset \cc S$, and the first line is, $(\Phi - \la\id)(I_n ) = (1-\la)I_n$ or ``$\Phi$ is unital.'' \qedhere
\end{proof}

\begin{prop}\label{prop:dual-primal} The dual of the primal program (\ref{eq:index-primal}) is:

\begin{equation}\label{eq:index-dual}
    \begin{aligned}
        &\textup{minimize} &&\tr(X)\\
        &\textup{s.t.} && Y\in M_n\otimes (\cc S)^\perp\\
        & && \tr\otimes\tr((I_n\otimes X + Y)\Delta_n)= 1 \\
        & && I_n\otimes X + Y\in (M_n\otimes M_n)^+.
    \end{aligned}
\end{equation}
The primal and dual programs compute the same value.
\end{prop}

\begin{proof} The proof is a routine computation, and proceeds similarly to the proof of \cite[Theorem 9]{Winter2013}. For ease of notation, given hermitian matrices $X,Y\in M_n$ we write $X\bullet Y$ to denote $\tr(XY)$. We again refer the reader to \cite[section 2]{AGKS} for background on complex semidefinite programming. Consider $M_{n^2+1}$ which we will give the basis indexing the entries by $F_{ij}$ where $(i,j)\in \{(0,0), \dotsc, (0, n), (1,0),\dotsc, (n,0)\}$ and $E_{ij}\otimes E_{kl}$ for $i,j,k,l\in \{1\dotsc,n\}$. We will consider the pair $(X,\la)$ in the primal program (\ref{eq:index-primal}) as a matrix in $M_{n^2+1}$ by $(X,\la)\mapsto \begin{pmatrix} \la & 0\\ 0 & X\end{pmatrix} = \la F_{00} + X$. The objective function thus becomes $F_{00}\bullet \tilde X$ where $\tilde X\in (M_{n^2+1})^+$. Let $G_1,\dotsc, G_m\in M_n$ be an orthonormal basis of hermitian matrices for $\cc S^\perp$. The constraints in (\ref{eq:index-primal}) then become:

\begin{equation}
    \begin{aligned}
        (F_{00} + I_n\otimes E_{ii})\bullet \tilde X = 1&, && i=1,\dotsc, n\\
        (I_n\otimes E_{ij})\bullet \tilde X = 0&, && i,j=1,\dotsc,n, \ i\not=j\\
        ([G_k]_{ij} F_{00} + E_{ij}\otimes G_k)\bullet \tilde X = 0&, && i,j=1,\dotsc,n,\ k=1,\dotsc,m.
    \end{aligned}
\end{equation}

The dual of (\ref{eq:index-primal}) thus becomes:

\begin{equation}
    \begin{aligned}
        & \text{minimize} && \tr(Y)\\
        & \text{subject to} && \tr(Y)F_{00} + I_n\otimes Y + \sum_{i,j,k} [Z_k]_{ij}[G_k]_{ij}F_{00} + \sum_k Z_k\otimes G_k - F_{00} \succeq 0.
    \end{aligned}
\end{equation}  
where $Y, Z_1,\dotsc, Z_m\in M_n$ are hermitian. Noting that $\tr(Y) = (I_n\otimes Y)\bullet \Delta_n$ and $\sum_{i,j} [Z_k]_{ij} [G_k]_{ij} = (Z_k\otimes G_k)\bullet \Delta_n$, we see the constraint in the previous program block decomposes into two constraints:

\begin{equation}
    \begin{aligned}
        & (I_n\otimes Y + \sum_{k} Z_k\otimes G_k)\bullet \Delta_n\geq 1\\
        & I_n\otimes Y + \sum_k Z_k\otimes G_k\succeq 0.
    \end{aligned}
\end{equation}
This verifies that the program (\ref{eq:index-dual}) is the dual of the program (\ref{eq:index-primal}). 

Since the Choi matrix of $X\mapsto \tr(X)I_n - \la X$ is positive definite for some $\la>0$, it follows from the Strong Duality Theorem for semidefinite programming (see \cite[Section 2]{AGKS} and the references therein) that the dual obtains the same value as the primal program. \qedhere
\end{proof}

\begin{remark}
    Proposition \ref{prop:dual-primal} offers an alternate proof of multiplicativity of the $\cp$-index for matricial systems. Indeed, let $\cc S\subset M_n$ and $\cc T\subset M_k$ be matricial systems. By Corollary \ref{cor:min-submultiplicative} it suffices to check that \[\Ind_{\cp}(M_{nk}:\cc S\otimes\cc T)\geq \Ind_{\cp}(M_n:\cc S)\cdot\Ind_{\cp}(M_k:\cc T).\]
    Let $\theta: M_n\otimes M_n\otimes M_k\otimes M_k\to M_{nk}\otimes M_{nk}$ be the shuffle isomorphism $\theta(A\otimes B\otimes C\otimes D) = A\otimes C\otimes B\otimes D$. If $X\in M_n$ and $Y\in M_n\otimes M_n$ satisfy the constraints for (\ref{eq:index-dual}) with respect to $\cc S$ and $Z\in M_k$ and $W\in M_k\otimes M_k$ do the same for $\cc T$, then it is easy to check that $X\otimes Z\in M_{nk}$ and $\theta(I_n\otimes X\otimes W + Y\otimes I_k\otimes Z + Y\otimes W)$ satisfy the constraints of (\ref{eq:index-dual}) with respect to $\cc S\otimes \cc T$. From this it follows that $\Ind_{\cp}(M_{nk}:\cc S\otimes \cc T)^{-1}\leq \tr(X)\tr(Z)$, which establishes the result.
\end{remark}

We now turn to describing the Lov\'asz theta invariant for a graph and its quantum version defined for matricial systems. Let $\G = (V, E)$ be a graph. We say that $v,w\in V$ are \emph{adjacent}, written $v\sim w$, if either $v=w$ or $(v,w)\in E$.
Consider a graph $\G$ on $n$ vertices which we will label as $1,\dotsc,n$. By a \emph{orthonormal representation} of the $\G$ we mean a set of unit vectors $\{x_1,\dotsc,x_n\}$ in a Euclidean space, such that if $i \nsim j$ then $x_i \perp x_j.$ We then define the \emph{value} of an orthonormal representation $\{x_1,\dotsc,x_n\}$ to be 
    \begin{align}
        \min_c \max_{1 \leq i \leq n} \frac{1}{(c^Tx_i)^2},
    \end{align} 
where $c$ ranges over all unit vectors. The vector $c$ yielding the minimum is called the \emph{handle} of the representation. The \emph{(classical) Lov\'asz theta} of the graph $\G$, $\vartheta(\G)$, as introduced in \cite{Lovasz1979}, is then defined as the minimum value over all representations of the graph $\G$. 

To a graph $\G$ on $n$ vertices we associate two matricial systems $\cc E_n\subset\cc S_\G\subset M_n$ defined by
    \begin{align}
        \label{eq:S-gamma} &\cc S_\G := \{X\in M_n : X_{ij}=0, i\not\sim j\},\\
        \label{eq:En} &\cc E_n := \{A\in M_n : A_{ii} = A_{jj},\ i,j=1,\dotsc,n\},\ \textup{and}\\
        \label{eq:E-Gamma} &\cc E_\G := \{A\in \cc E_n : A_{ij}=0, i\not\sim j\}.
    \end{align}
Given a graph $\G = (V,E)$ we define its \emph{complement} $\bar\G$ to be the graph $\bar\G = (V, \bar E)$, where $\bar E$ is the complement of the edge set. Notice that 
\begin{equation}\label{eq:E-S-complememt}
    \cc E_\G = \bb CI_n + (\cc S_{\overline\G})^\perp.
\end{equation}
It follows from \cite[Theorem 3.6.1]{GartnerMatousek2012} that
    \begin{equation}\label{eq:Lov\'asz-theta}
        \vartheta(\G) = \min\{A_{11} : A\in \cc E_\G,\ A\succeq J_n\} = \min\{\max_i B_{ii} : B\in \cc S_\G,\ B\succeq J_n\},
    \end{equation} where $J_n$ is the matrix with all entries equal to $1$. (See also \cite[Corollary 12]{Winter2013}.)
    
\begin{prop}\label{prop:cp-Lov\'asz-equal}
    For a graph $\G$ on $n$ vertices we have that 
    \[\Ind_{\cp}(M_n:\cc S_\G) = \Ind(M_n:\cc E_\G) =\vartheta(\G)\ \textup{and}\ \tilde\la(\cc E_{\G}) = \vartheta(\overline\G).\]
\end{prop}

Before proving this result we require one standard lemma.

\begin{lem}\label{lem:choi-expectation}
    If $\phi: M_n\to M_n$ is a completely positive map, then the matrix $A_{ij} := \phi(E_{ij})_{ij}$ is positive semidefinite. Moreover, $\max_i A_{ii} \leq \|\phi(I_n)\|$. 
\end{lem}

\begin{proof}
    The map $\Delta: E_{ij}\mapsto E_{ij}\otimes E_{ij}$ induces a (non-unital) $\ast$-embedding of $M_n$ into $M_n\otimes M_n$. This implies that $\Delta(B) = \sum_{ij} B_{ij}E_{ij}\otimes E_{ij}$ is positive semidefinite for all $B\in M_n$ positive semidefinite. We see that $0\leq \tr((\Ch(\phi)\circ \Delta(B))\Delta(J_n)) = \tr(AB)$ for all $B\in M_n$ positive semidefinite; thus, $A$ is positive semidefinite. The second assertion follows since $A_{ii} = \phi(E_{ii})_{ii} \leq \|\phi(E_{ii})\|\leq \|\phi(I_n)\|$. \qedhere
\end{proof}


\begin{proof}[Proof of Proposition \ref{prop:cp-Lov\'asz-equal}]
    We begin by showing that $\Ind_{\cp}(M_n:\cc S_\G) =\vartheta(\G)$. Let $\phi: M_n\to \cc S_\G$ in $\cp_1(M_n)$ be such that $\psi := \phi - \id$ is completely positive. Setting
    \[A_{ij} = \phi(E_{ij})_{ij}\quad \textup{and}\quad B_{ij} = \psi(E_{ij})_{ij},\]
    by Lemma \ref{lem:choi-expectation} we have that $B = A - J_n\succeq 0$ and $A_{ii}\leq \|\phi(I_n)\|$; thus 
    \[\vartheta(\G)\leq \Ind_{\cp}(M_n:\cc S_\G).\] 
    
    The reverse inequality is obtained by considering the Schur multiplier $\de_A$ of any matrix $A\in \cc E_\G$ as in (\ref{eq:Lov\'asz-theta}). We have that $\de_A(I_n) = A_{11} I_n$, the image of $\de_A$ is contained in $\cc S_\G$, and that $\de_{A-J_n} = \de_A-\id$ is completely positive. It follows that
    \[\Ind_{\cp}(M_n: \cc S_\G) \leq \Ind_{\cp}(M_n:\cc E_\G) \leq \vartheta(\G).\]

    We now turn to the second assertion. For ease of notation, set $\cc S = \cc E_{\overline \G} =\bb CI_n + \cc S_\G^\perp$. Let $\phi:\cc S\to \bb CI_n$ be such that $\psi:=\phi-\id_{\cc S}: \cc S\to \cc S$ is completely positive. We may extend $\psi$ to a completely positive map $\tilde\psi: M_n\to M_n$, and define $\tilde\phi := \tilde\psi + \id_{M_n}$, noting that $\tilde\phi$ extends $\phi$. Define $A,B\in M_n$ by 
    \[A_{ij} := \tilde\phi(E_{ij})_{ij}\quad \textup{and}\quad B_{ij} := \tilde\psi(E_{ij})_{ij},\] 
    which are positive semidefinite by Lemma \ref{lem:choi-expectation}.
    We note that $A - J_n = B$, $A_{ii} = \la$ if $\phi(I_n) = \la I_n$, and that $A_{ij} = 0$ if $i\not\sim j$. It therefore follows from (\ref{eq:Lov\'asz-theta}) that $\vartheta(\G)\leq \tilde\la(\cc S)$.
    
    In the other direction, suppose that $A\in \cc E_\G$ is such that $A\succeq J_n$. Let $\de_A$ be the Schur multiplier associated to $A$, and define $\phi:M_n\to M_n$ by $\phi(E_{ij}) = \de_A(E_{ij}) = A_{ij}E_{ij}$ if $i\not= j$ and $\phi(E_{ii}) = A_{ii}I_n$ for all $i=1,\dotsc,n$. The complete positivity of $\phi - \id_{M_n}$ follows from that of $\de_A - \id_{M_n} = \de_{A - J_n}$, and we have that $\phi(\cc S)\subset \bb CI_n$. Thus $\vartheta(\G)\geq \tilde\la(\cc S)$.  \qedhere
    
\end{proof}

\begin{remark}
     It ought to follow by similar arguments that $\vartheta(\overline{\G}) = \Ind_{\cp}(\cc S_\G:D_n)$, though we leave this unresolved.
\end{remark}

\begin{remark}
    Given an inclusion $\cc T\subset \cc S\subset M_n$, one can define a relative version of the quantum Lov\'asz invariant by $\widetilde\vartheta(\cc S : \cc T) := \Ind_{\cp}(\cc S: \cc T)$. If $\La$ is a subgraph of a graph $\G$ on $n$ vertices, one obtains a relativized version of the Lov\'asz theta invariant as $\vartheta(\G : \La) = \Ind_{\cp}(\cc S_\G: \cc S_\La)$. By the proof of Proposition \ref{prop:cp-Lov\'asz-equal} have that $\vartheta(\G : \La)$ is still semidefinite programmable and that $\vartheta(\G:\La)\leq \vartheta(K_n:\La) = \vartheta(\La)$. One would expect strict inequality in some cases where $\cc S_\La$ is not injective as an operator system. It would be interesting if this could be used practically to give sharpened estimates on the chromatic number or other graph invariants.
\end{remark}

Duan, Severini, and Winter defined an invariant for matricial systems $\cc S\subset M_n$ known as the \emph{quantum Lov\'asz theta} $\tilde\vartheta(\cc S)$. We recall the following formulation for $\widetilde\vartheta(\cc S)$ as a complex semidefinite program which appears as \cite[Theorem 8]{Winter2013}:
\begin{equation}\label{eq:dsw-theta}
    \begin{aligned}
        &\text{maximize} && \tr\otimes\tr((I_n\otimes X + Y)\Delta_n)\\
        &\text{subject to} && \tr(X) =1\\
        & && Y\in (\cc S)^\perp\otimes M_n\\
        & && I_n\otimes X + Y\in (M_n\otimes M_n)^+.
    \end{aligned}
\end{equation}
From (\ref{eq:dsw-theta}) it is easy to see by a change of variables that the following program computes $\widetilde\vartheta(\cc S)^{-1}$:
\begin{equation}\label{eq:dsw-theta-inv}
    \begin{aligned}
        &\text{minimize} && \tr(X)\\
        &\text{subject to} && \tr\otimes\tr((I_n\otimes X + Y)\Delta_n)= 1\\
        & && Y\in (\cc S)^\perp\otimes M_n\\
        & && I_n\otimes X + Y\in (M_n\otimes M_n)^+.
    \end{aligned}
\end{equation}

Comparing this with (\ref{eq:index-dual}), we observe that the difference between $\Ind_{\cp}(M_n : \cc S)$ and $\widetilde\vartheta(\cc S)$ is that, all other things being equal, in the former we have that $Y\in M_n\otimes (\cc S)^\perp$ while in the latter $Y\in (\cc S)^\perp\otimes M_n$. The following proposition is then immediate from the proofs of Propositions \ref{prop:cp-ind-primal} and \ref{prop:dual-primal}.

\begin{prop}[cf.\ Theorem 9 in \cite{Winter2013}]
    The following semidefinite program computes $\vartheta(\cc S)^{-1}$:
    \begin{equation}\label{eq:theta-primal}
        \begin{aligned}
            &\textup{maximize} &&\la \\
            &\textup{subject to} &&\id\otimes\tr(X) = (1-\la)I_n\\
            & && X +\la\Delta_n\in M_n\otimes\cc S\\
            & && X\in (M_n\otimes M_n)^+.
        \end{aligned}
    \end{equation}
    Equivalently, $\widetilde\vartheta(\cc S)^{-1}$ is the maximal value $\la\geq 0$ so that there is a completely positive, trace-preserving map $\Phi: M_n\to \cc S$ so that $\Phi - \la\id_{M_n}$ is completely positive.
\end{prop}

For a completely positive map $\Phi: M_n\to M_n$, we may define a \emph{dual} map $\Phi^\dagger: M_n\to M_n$ by the functional equation
\begin{equation*}
    \tr(\Phi(X)Y) = \tr(X\Phi^\dagger(Y)),\quad X,Y\in M_n.
\end{equation*}
It is well known and not difficult to check that $\Phi$ is unital, completely positive if and only if $\Phi^\dagger$ is completely positive and trace preserving. From this we obtain the following corollary.

\begin{cor}
    We have that $\widetilde\vartheta(\cc S)^{-1}$ is the maximal value $\la\geq 0$ so that there is a unital, completely positive map $\Phi: M_n\to \cc S$ so that $\Phi(\cc S^{\perp}) = \{0\}$ and $\Phi - \la\id_{M_n}$ is completely positive.
\end{cor}
    
It is interesting to compare this to the co-index of $\cc S^\perp$ in $M_n$ defined in Section \ref{sec:other-notions-index} below and to Proposition \ref{prop:co-index-2} in particular.

\begin{question}
    For an arbitrary matricial system $\cc S\subset M_n$, are there any equalities among $\Ind_{\cp}(M_n:\cc S)$, $\widetilde\vartheta(\cc S)$, and $\tilde\la(\bb C 1 + \cc S^\perp)$?
\end{question}

\subsection{Computations of indices of group generator systems} 
We begin by looking at operator systems in the full group $C^*$-algebra  $C^*(\bb F_n).$

For a graph $\G$ on $n$ vertices, let $\cc T_\G\subset \cc T_n\subset C^*(\bb F_n)$ be defined by:
\begin{equation*}
    \begin{aligned}
        \cc T_n &= {\rm span}\{u_i^*u_j : i,j=1,\dotsc,n\}\\
        \cc T_\G &={\rm span}\{u_i^*u_j : i\sim j\}
    \end{aligned}
\end{equation*}
where $u_1,\dotsc,u$ are the canonical unitary generators of $C^*(\bb F_n)$.

\begin{prop}
    We have that $\Ind_{\cp}(\cc T_n: \cc T_\G) = \vartheta(\G)$.
\end{prop}

\begin{proof}  
    We first prove that $\Ind_{\cp}(\cc T_n :\cc T_\G)\leq \vartheta(\G)$. 
    It follows from \cite[Theorem 2.6]{Farenick2012} that, for a matrix $A\in \cc E_n$, the map $\phi_A:\cc T_n\to \cc T_n$ given by $\phi_A(u_i^*u_j) = A_{ij}\,u_i^*u_j$ if $i\not= j$ and $\phi_A(1) = A_{11}1$ is completely positive if and only if $A$ is positive semidefinite. Also note that $\phi_A(\cc T_n)\subset \cc T_\G$ if and only if $A\in \cc E_\G$. From this and (\ref{eq:Lov\'asz-theta}) it follows that $\Ind_{\cp}(\cc T_n : \cc T_\G)\leq \vartheta(\G)$.

 By \cite[Proposition 2.9]{Farenick2012} we have that $\cc T_n^*$ is canonically completely order isomorphic to $\cc E_n$; hence, by \cite[Proposition 1.15]{Farenick2012} we have that $\cc T_\G^*\cong \cc E_n/(\cc E_\G^\perp\cap \cc E_n)$. We have $\phi - \id: \cc T_n\to \cc T_n$ is completely positive if and only if $\phi^* - \id = (\phi - \id)^*: \cc E_n\to \cc E_n$ is. If $\phi(\cc T_n) \subset \cc T_\G$, then $\ker(\phi^*)\supset \cc E_\G^\perp \cap \cc E_n$. We can extend $\phi^* - \id$ to a completely positive map $\psi: M_n\to M_n$ so that $\tilde\psi := \psi + \id$ extends $\phi^*$. 
 
 Let $A$ be the matrix obtained from $\tilde\psi$ as in Lemma \ref{lem:choi-expectation}, and let $\phi_A$ be the associated Schur multiplier. From Lemma \ref{lem:choi-expectation} we see that $\|\phi_A(1)\| = \max_i A_{ii} \leq \|\tilde\psi(1)\|$ and that $\phi_A - \id$ is completely positive. In this way we have that $\vartheta(\G)\leq \Ind_{\cp}(\cc T_n : \cc T_\G)$. \qedhere
    
\end{proof}

\begin{example}
    Let $G$ be a countable discrete group and $S =S^{-1} \subset G$ a finite, symmetric generating set. Let $C_\la^*(G)$ denote the reduced group C$^*$-algebra of $G$, and let $\cc X_S\subset C_{\la}^*(G)$ be the operator system spanned by $\{\la(s) : s\in S\cup\{e\}\}$. 
\end{example}

Let $E: C_\la^*(G)\to \bb C1$ be the conditional expectation onto $\bb C1$. Let $x\in \cc X_S\setminus\{0\}$ be self-adjoint and supported on $S\setminus\{e\}$, i.e., $E(x)=0$. We write
\[\rho(x) := \max\left\{\frac{a}{b}, \frac{b}{a}\right\}\]
where $a,b>0$ are the smallest constants such that $a1 \succeq x\succeq -b1$. We define
\[\rho(S) = \max\{\rho(x) : x^* = x\in \cc X_S\setminus\{0\},\ E(x)=0\}.\]

\begin{prop}\label{prop:ind-symmetric-generating}
    We have that
    \begin{equation*}
        \Ind(\cc X_S:\bb C1)\leq 1+\rho(S)
    \end{equation*}
\end{prop}

\begin{proof}
     We have that $\Ind(\cc X_S : \bb C 1)= \mu_\ast^{-1}$ where 
    \begin{equation*}
        \begin{split}
            \mu_\ast &= \max\{\mu\geq 0 : \phi\in \ucp(\cc X_S),\ \phi(\cc X_S)=\bb C1,\ \phi - \mu\id\succ\succ 0\}\\
            &\geq \max\{\mu\geq 0 : E - \mu\id\succ\succ 0\} =: \mu_0.
        \end{split}
    \end{equation*}
    Since every positive element in $\cc X_S$ is of the form $c1 + x$ with $c\geq 0$, $x^* = x$, and $E(x)=0$, in order to check positivity of the map $E - \mu_0\id$ it suffices to consider the only the case when $c$ is minimal so that $c1 + x$ is positive and $x\not= 0$. For a fixed self-adjoint $x$ with $E(x)=0$, consider $a1 - x, b1 + x\in \cc X_S^+$ with $a,b$ minimal. We have that
    \begin{equation}\label{eq:E-la}
        \begin{split}
            &E(a1 - x) -\mu_0(a1 - x) = (1-\mu_0)a1 + \mu_0x\succeq 0\\
            &E(b1 + x) - \mu_0(b1 + x) = (1-\mu_0)b1 - \mu_0x\succeq 0.
        \end{split}
    \end{equation}
    
    From the first line of (\ref{eq:E-la}) we have that $(1-\mu_0)a1\succeq -\mu_0x$ from which it follows from assumptions holds if and only if $(1-\mu_0)a \geq \mu_0b$. Similarly, $(1-\mu_0)b\geq \mu_0a$ if and only if $(1-\mu_0)b1\succeq \mu_0x$ follows from the second line of (\ref{eq:E-la}) and standing assumptions. In this way we see that 
    \[\mu_0\leq \min\left\{\frac{a}{a+b}, \frac{b}{a+b}\right\}\]
    with equality obtained by minimizing over all self-adjoint contractions $x\in \cc X_S$ with $E(x)=0$. Thus $\mu_\ast^{-1}\leq \mu_0^{-1} = 1 + \rho(S)$. \qedhere
\end{proof}

\begin{remark} 
    For a state $\phi:\cc X\to \bb C$, let $R(\phi) := \{x\succeq 0: \phi(x)\leq 1\}$. Define
    \begin{equation*}
        \rho(\cc X,\phi) := \sup\{\|x\| : x\in R(\phi)\}.
    \end{equation*}
    By essentially the same considerations as the previous result we have that 
        \[\Ind(\cc X:\bb C1) = \inf_\phi \rho(\cc X,\phi)\]
    where $\phi$ ranges over all states $\phi:\cc X\to \bb C$.
\end{remark}

\begin{prop}\label{prop:cp-hoffman}
    We have that $\tilde\la(\cc X_S)\geq 1 + \rho(S)$.
\end{prop}

\begin{proof}
    The argument here is essentially due to Haagerup \cite[Lemma 2.5]{Haagerup}. Let $\phi: \cc X_S\to \cc X_S$ be a completely positive map. We may describe $\phi$ in terms of the ``matrix coefficients'' $\phi(g,h)$ where
    \begin{equation}
        \phi(\la(g)) = \sum_{h\in S} \phi(g,h) \la(h).
    \end{equation}
    We claim that the multiplier map $m_\phi(\la(g)) := \phi(g,g) \la(g)$ is completely positive.
    
    Indeed, consider the comultiplication map \[\Delta: C_\la^*(G)\to C_\la^*(G\times G)\cong C_\la^*(G)\otimes_{\min} C_\la^*(G)\] given by $\Delta(\la(g)) = \la(g)\otimes\la(g)$. Since $\Delta$ is an injective $\ast$-homomorphism, we have that the restriction $\Delta_S: \cc X_S\to \cc X_S\otimes_{\min} \cc X_S$ is a complete order embedding. Moreover, since $\Delta(G)$ is a subgroup of $G\times G$, there is a conditional expectation $E_{\Delta}: C_\la^*(G\times G)\to \Delta(C_\la^*(G))$. It can now be checked that $m_\phi = \Delta_S^{-1}\circ E_\Delta\circ(\id\otimes\phi)\circ\Delta_S$; hence, $m_\phi$ is completely positive.
    
    Now, let $\phi:\cc X_S\to \bb C1$ be such that $\psi :=\phi - \id:\cc X_S\to \cc X_S$ is completely positive. It follows from the previous paragraph that $m_\psi = m_\phi - m_{\id} = \phi(e,e) E - \id$ is completely positive; thus, 
    \begin{equation}
        \tilde\la(\cc X_S) = \inf\{\la : \la\cdot E - \id\in \cp(\cc X_S)\}\geq \inf\{\la : \la\cdot E \succ\succ \id\} = 1 + \rho(S),
    \end{equation}
    where the final equality is given in the proof of Proposition \ref{prop:ind-symmetric-generating}. \qedhere
\end{proof}

\begin{question}
    Does $\tilde{\la}(\cc X_S) = \Ind(\cc X_S:\bb C1) = 1 + \rho(S)$?
\end{question}

\begin{remark}
    Generalizing \cite[Theorem 6]{Lovasz1979}, for a matricial system $\cc S\subset M_n$ the Lov\'asz theta invariant $\vartheta(\cc S)$ is defined in \cite[Section IV]{Winter2013} as
    \begin{equation*}
        \vartheta(\cc S) = \max\left\{1 + \frac{\la_{\max}(x)}{|\la_{\min}(x)|} : x=x^*\in \cc S^\perp\right\}.
    \end{equation*}
    Thus, Propositions \ref{prop:ind-symmetric-generating} and \ref{prop:cp-hoffman} provide another intriguing connection between the Lov\'asz theta invariant and the index. Conceivably, $\tilde{\la}(\cc X)$ could provide a generalization of a Hoffman-type bound \cite{Hoffman1970, Ganesan} to some notion of a ``quantum chromatic number'' \cite{PaulsenTodorov2015} of the operator system $\cc X$.
\end{remark}

In light of this connection it would be interesting to know whether $\tilde\la(\cc X_{S^n})\leq \tilde\la(\cc X_S)^n$ for all $n\in \bb N$, where $S^n$ is the $n$-fold product of $S$. We would then have that
\[\tilde\alpha(S) := \lim_{n\to \infty} \tilde\la(\cc X_{S^n})^{1/n}\] would exist and should serve as some analog for the Shannon capacity, as in \cite{Lovasz1979}, for the Cayley graph of $G$ induced by $S$.

\section{Other Notions of Index}\label{sec:other-notions-index}

\subsection{The Co-index}

\begin{defn}
Given an operator system $\cc X$ and a kernel $\cc J \subset \cc X$ then we define the \emph{Co-index} of the quotient operator system $\cc X/ \cc J$ in $\cc X$ to be 
\begin{align*}
    \coind(\cc X: \cc J):= \Ind_{*}(\cc X^*: (\cc X/ \cc J)^*),
\end{align*} where 
\begin{align*}
    \Ind_{*}(\cc X^*: (\cc X/ \cc J)^*):= \inf \{ \cbnorm{\vp}: \vp \in \CP(\cc X^*; (\cc X/ \cc J)^*), \vp - \id_{\cc X^*} \in \CP(\cc X^*)\}.
\end{align*}
\end{defn}

\begin{lem}
Let $\al$ be an operator system tensor product. Then the dual tensor product $\al^*$ is also an operator system tensor product relative to finite-dimensional operator systems. 
\end{lem}

\begin{proof}
First note $\al^*(\cc X, \cc Y):= (\cc X^* \otimes_\al \cc Y^*)^*$ is an operator system, being the Banach dual of a finite-dimensional operator system. Suppose $x \in M_n(\cc X)^+, y \in M_m(\cc Y)^+.$ By identifying $\cc X \simeq \cc X^{**}$ and $\cc Y \simeq \cc Y^{**}$ it follows $x \in M_n(\bidual{\cc X})^+$ and $y \in M_m(\bidual{\cc Y})^+$. This implies the induced maps \begin{align*}
    x: \cc X^* \to M_n, \quad \text{and}\quad y: \cc Y^* \to M_m,
\end{align*} are both completely positive. Since $\al$ is an operator system tensor product it follows the induced linear map \[
x \otimes y: \cc X^* \otimes_\al \cc Y^* \to M_{mn},
\] is completely positive. In particular, $x \otimes y \in M_{mn}((\cc X^* \otimes_\al \cc Y^*)^*)^+ = M_{mn}(\cc X \otimes_{\al^*} \cc Y)^+.$ This proves the second axiom for tensor products in \text{OpSys}. In order to check matrix functoriality, let $u: \cc X \to M_n$ and $v: \cc Y \to M_m$ both be unital completely positive. It then follows since $u \in M_n(\cc X^*)^+$ and $v \in M_m(\cc Y^*)^+$ that $u \otimes v \in M_{mn}(\cc X^* \otimes_{\al} \cc Y^*)^+.$ Embedding into the corresponding bidual operator system we have 
\[u \otimes v \in M_{mn}(\bidual{(\cc X^* \otimes_{\al} \cc Y^*)})^+ = M_{mn}((\cc X \otimes_{\al^*}\cc Y)^*)^+.
\] We conclude the linear map $u \otimes v: \cc X \otimes_{\al^*} \cc Y \to M_{mn}$ is unital completely positive. This proves $\al^*$ satisfies all three axioms and is therefore a tensor product in \text{OpSys}. \qedhere
\end{proof}

\begin{remark}\label{remark: dual tensor product is functorial relative to completely positive maps}
Consider finite-dimensional operator systems $\cc X, \cc Y, \cc S,\cc T$ and let $\al$ be an operator system tensor product which is functorial relative to completely positive maps. Given completely positive maps $u: \cc X \to \cc S$ and $v: \cc Y \to \cc T$ it needs to be shown that they induce a completely positive map on the dual tensor product $\al^*$. This is proven in the following way: first consider the dual maps $u^*: \cc S^* \to \cc X^*$ and $v^*: \cc T^* \to \cc Y^*$, which are necessarily completely positive. Functoriality of $\al$ implies the induced linear map $u^* \otimes v^*: \cc S^* \otimes_{\al} \cc T^* \to \cc X^* \otimes_{\al} \cc Y^*$ is completely positive. Taking the dual of this map yields a completely positive map $(u^* \otimes v^*)^*: (\cc X^* \otimes_{\al} \cc Y^*)^* \to (\cc S^* \otimes_{\al} \cc T^*)^*$. Making the identifications $\cc X \simeq \bidual{\cc X}, \cc Y \simeq \bidual{\cc Y}$ and observing $(u^* \otimes v^*)^* = u \otimes v$ under this identification, implies we have a completely positive map \[
u \otimes v: \cc X \otimes_{\al^*} \cc Y \to \cc S \otimes_{\al^*} \cc T.
\]
\end{remark}

\begin{prop}
    Let $\cc X$ and $\cc Y$ be two finite-dimensional operator systems and let $u: \cc X \to u(\cc X)$ and $v: \cc Y \to v(\cc Y)$ be two complete quotient maps. If $\al$ is a projective operator system tensor product which is functorial relative to completely positive maps, then it follows \begin{align*}
        \coind(\cc X \otimes_{\al} \cc Y: \cc N(u \otimes v)) \leq \coind(\cc X: \cc N(u)) \coind(\cc Y: \cc N(v)). 
    \end{align*} 
    
\end{prop}

\begin{proof}
Consider $\vp \in \CP(\cc X^*; (\cc X/ \cc N(u))^*), \psi \in \CP(\cc Y^*; (\cc Y/ \cc N(v))^*)$ such that $\vp - \id_{\cc X^*} \in \CP(\cc X^*)$ and $\psi - \id_{\cc Y^*} \in \CP(\cc Y^*).$ By Remark~\ref{remark: dual tensor product is functorial relative to completely positive maps} we know the induced linear map \begin{align*}
        \vp \otimes \psi: \cc X^* \otimes_{\al^*} \cc Y^* \to (\cc X/ \cc N(u))^* \otimes_{\al^*} (\cc Y/ \cc N(v))^*,
     \end{align*} is completely positive. In particular, after identifying the respective operator systems with their biduals, and using projectivity of $\al$, we obtain a completely positive map \begin{align*}
         \vp \otimes \psi: (\cc X \otimes_{\al} \cc Y)^* \to (\cc X \otimes_{\al} \cc Y/ \cc N(u \otimes v))^*.
     \end{align*} By the same reasoning as in Proposition~\ref{prop: CP-index is submultiplicative}, we deduce $\vp \otimes \psi - \id_{(\cc X \otimes_{\al} \cc Y)^*} \in \CP((\cc X \otimes_{\al} \cc Y)^*).$ Indeed, since $\vp - \id_{\cc X^*} \in \CP(\cc X^*)$ and $\psi - \id_{\cc Y^*} \in \CP(\cc Y^*)$, and by applying Remark~\ref{remark: dual tensor product is functorial relative to completely positive maps}, we deduce \[(\vp \otimes \id_{\cc Y^*})- (\id_{\cc X^*} \otimes \id_{\cc Y^*}) \in \CP(\cc X^* \otimes_{\al^*} \cc Y^*),\] and similarly, \[(\id_{\cc X^*} \otimes \psi) - (\id_{\cc X^*} \otimes \id_{\cc Y^*}) \in \CP(\cc X^* \otimes_{\al^*} \cc Y^*).\] We conclude, \[
     \vp \otimes \psi \geq \vp \otimes \id_{\cc Y^*} \geq \id_{\cc X^*} \otimes \id_{\cc Y^*}.
     \] Since we have the vector space identification $\cc X^* \otimes \cc Y^* \simeq (\cc X \otimes \cc Y)^*$ this proves \begin{align*}
         \vp \otimes \psi - \id_{(\cc X \otimes_{\al} \cc Y)^*} \in \CP((\cc X \otimes_{\al} \cc Y)^*).
     \end{align*}
     
     Let \begin{align*}
         A&:= \{\vp \in \CP(\cc X^*; (\cc X/ \cc N(u))^*):  \vp - \id_{\cc X^*} \in \CP(\cc X^*) \},\,\,\text{and} \\
         B&:= \{ \psi \in \CP(\cc Y^*; (\cc Y/ \cc N(v))^*): \psi - \id_{\cc Y^*} \in \CP(\cc Y^*)\}. 
     \end{align*}
We observe \begin{align*}
    \coind&(\cc X \otimes_\al \cc Y: \cc N(u\otimes v)) \\
    &\leq \inf \{\cbnorm{\vp \otimes \psi}: \vp\in A, \psi \in B\} \\
    &\leq \inf \{ \cbnorm{\vp} \cbnorm{\psi}: \vp \in A, \psi \in B\} \\
    &= \inf \{ \cbnorm{\vp}: \vp \in A \} \inf \{\cbnorm{\psi}: \psi \in B\} \\
    &= \coind(\cc X: \cc N(u)) \coind(\cc Y: \cc N(v)).
\end{align*} This finishes the proof. \qedhere
\end{proof}

For operator systems $\cc X$ and $\cc Y$ with kernels $\cc J\subset \cc X$ and $\cc K\subset \cc Y$, we define $\cc J\vee \cc K\subset \cc X\otimes_{\max} \cc Y$ to be the smallest kernel containing $\cc J\otimes \cc Y + \cc X\otimes \cc K$. The following corollary is immediate:

\begin{cor}
    For finite-dimensional operator systems $\cc X$ and $\cc Y$ with kernels $\cc J\subset \cc X$ and $\cc K\subset \cc Y$ we have that
    \begin{equation}
        \coind(\cc X\otimes_{\max} \cc Y: \cc J\vee \cc K)\leq \coind(\cc X: \cc J)\coind(\cc Y: \cc K).
    \end{equation}
\end{cor}

We describe another way to approach the co-index in the finite-dimensional case. Let $\cc X$ be a finite-dimensional operator system. We have that $\vp \leftrightarrow \vp^*$ is a complete isometry between $\CB(\cc X)$ and $\CB(\cc X^*)$ which sends $\cp(\cc X)$ onto $\cp(\cc X^*)$ by \cite[Lemma 1.4]{Kavruk2014}. In this way, following the argument given in \cite[Proposition 2.7]{Kavruk2014}, the cone $\CP(\cc X^*, (\cc X/\cc J)^*)$ is completely isometrically identified with the cone \[\{\vp\in \cp(\cc X) : \ker(\vp)\supset \cc J\}.\]
In this way we have proven
\begin{prop}\label{prop:co-index-2}
    Let $\cc X$ be a finite-dimensional operator system, and let $\cc J\subset \cc X$ be a kernel. We have that
    \begin{equation}
        \coind(\cc X: \cc J) = \inf\{\|\vp\|_{\cb} : \vp\in\cp(\cc X),\  \ker(\vp)\supset \cc J,\ \vp-\id_{\cc X}\in \cp(\cc X)\}.
    \end{equation}
\end{prop}

\begin{example}
    Let $G = (V,E)$ be a graph, and let $\cc S_\G\subset M_n$ be defined as in (\ref{eq:S-gamma}) above. We have that $S_\G^\perp\subset M_n$ is a kernel. Let $A\in M_n$ be positive semidefinite, and consider the Schur multiplier $\vp_A\in \cp(M_n)$. We have that $\ker(\vp_A)\subset S_\G^\perp$ if and only if $A_{ij}=0$ for all $i\not\sim j$. We have that $\vp_A - \id$ is completely positive if and only if $A - J_n$ is positive semidefinite. Thus by (\ref{eq:Lov\'asz-theta}) we have that 
    \begin{equation*}
        \vartheta(\G) \geq \coind(M_n:\cc S_\G^\perp).
    \end{equation*}
    It would be interesting to know if these two quantities are equal.
\end{example}

\subsection{The CB-index}

\begin{defn}\label{defn: cb index}
Given an inclusion of operator spaces $\cc X_0 \subset \cc X$ we define the \emph{CB-index} as $\Ind_\text{CB}(\cc X: \cc X_0):= \lambda_{\cc X}$ where 
\begin{equation}
    \la_{\cc X} = \inf\{\|u\|_{\cb} : u\in \CB(\cc X, \cc X_0),\ \|u - \id_{\cc X}\|_{\cb} \leq \|u\|_{\cb}-1\}.
\end{equation}
Equivalently, we have that 
\begin{align} \lambda_{\cc X}^{-1} = \sup \,\,\{\lambda \in (0,1]: \exists\,u \in \CB(\cc X, \cc X_0),\ \cbnorm{u}=1,\ \cbnorm{u - \lambda \id_{\cc X}} \leq 1 - \lambda\}.
\end{align}
\end{defn}

\begin{prop}
    Let $\cc X_0 \subset \cc X$ and $\cc Y_0 \subset \cc Y$ be inclusions of operator spaces and let $\al$ be a reasonable operator space tensor norm. Then if $\cc R:= \cc X \otimes_\al \cc Y$ and $\cc R_0:= \cc X_0 \otimes_\text{rel} \cc Y_0$ then \begin{align*}
        \Ind_{\CB}(\cc R: \cc R_0) \leq \Ind_{\CB}(\cc X: \cc X_0) \Ind_{\CB}(\cc Y: \cc Y_0).
    \end{align*}
\end{prop}

\begin{proof}
    Let $u:\cc X\to \cc X_0$ and $v:\cc Y\to \cc Y_0$ be completely bounded maps with $\cbnorm{u}=\cbnorm{v}=1$, and let $\la,\mu\geq 0$ be such that $\cbnorm{u - \la\id_{\cc X}}\leq 1-\la$ and $\cbnorm{v - \mu\id_{\cc Y}}\leq 1 - \mu$.
    Since $\alpha$ is reasonable, we have that $\cbnorm{u\otimes v: \cc R\to \cc R_0} =1$.
    Consider the maps
    \[\vp := u\otimes v - \la\id_{\cc X}\otimes v,\quad \psi := \id_{\cc X}\otimes v - \mu\id_{\cc R}\in \CB(\cc R).\]
    We have that $\cbnorm{\vp}\leq 1-\la$ and $\cbnorm{\psi}\leq 1-\mu$. Since $u\otimes v- \la\mu\id_{\cc R} = \vp - \la\psi$, it follows from the triangle inequality that
    \begin{equation*}
        \cbnorm{u\otimes v - \la\mu\id_{\cc R}}\leq (1-\la) + \la(1-\mu) = 1 - \la\mu 
    \end{equation*}
    This verifies that $\Ind_{\CB}(\cc R: \cc R_0)\leq (\la\mu)^{-1}$, which proves the result. \qedhere
\end{proof}

Let $\cc X$ be an operator space. Given a concrete realization $\cc X\subset \cc B(H)$, we can define the \emph{Paulsen system}, $\widetilde{\cc X}$, as the operator system
\begin{equation*}
    \widetilde{\cc X} := \left\{\begin{pmatrix} \eta_1 1 & x\\ y^* & \eta_2 1\end{pmatrix} : x,y\in \cc X,\ \eta_1,\eta_2\in \bb C\right\}\subset M_2(\cc B(H)).
\end{equation*}

Given a linear map $u: \cc X\to \cc Y$ between operator spaces and $R\geq 0$, we can induce a $\ast$-linear map $\tilde u_R : \widetilde{\cc X}\to \widetilde{\cc Y}$ on the the respective Paulsen systems by 
\begin{equation*}
    \tilde u_R\begin{pmatrix} \eta_1 1 & x\\ y^* & \eta_2 1\end{pmatrix} = \begin{pmatrix} R\,\eta_1 1 & u(x)\\ u(y)^* & R\,\eta_2 1\end{pmatrix}.
\end{equation*}
By construction $\tilde u_R$ sends any scalar multiple of the identity in $\widetilde{\cc X}$ to a scalar multiple of the identity.

The following lemma is well known: see \cite[Theorem B.5]{BrownOzawa} for a proof.

\begin{lem}\label{lem:paulsen-cp-map}
    The map $\tilde u_R$ is completely positive if and only if $R\geq \cbnorm{u}$.
\end{lem}

From this result we can deduce the following

\begin{lem}\label{lem:cb-to-cp-index}
    Let $\cc X$ be an operator space. For $u\in \CB(\cc X)$, we have that 
    \[\cbnorm{u - \la\id_{\cc X}} \leq 1-\la\ \text{for}\ \la\geq 0\] if and only if $\tilde u_1- \la\id_{\widetilde{\cc X}}\in \cp_1(\widetilde{\cc X})$.
\end{lem}

\begin{proof}
    Setting $v := u - \la\id_{\cc X}\in \CB(\cc X)$, by Lemma \ref{lem:paulsen-cp-map} we have that $\cbnorm{v}\leq 1- \la$ if and only if $\widetilde{v}_{(1 - \la)}$ is completely positive. A standard computation shows that $\tilde{v}_{(1-\la)} = \tilde u_1 - \la\id_{\widetilde{\cc X}}$. \qedhere
\end{proof}

\begin{prop}
Let $\cc X_0 \subset \cc X$ be an inclusion of operator spaces and let $\widetilde{\cc X}_o\subset \widetilde{\cc X}$ be the corresponding operator system inclusion of Paulsen systems. Then 
\begin{align*}
    \Ind_{\cp}(\widetilde{\cc X}: \widetilde{\cc X}_o) \leq \Ind_{\CB}(\cc X: \cc X_0).
\end{align*}
\end{prop}

\begin{proof}
    Let $u\in \CB(\cc X)$ be such that $\cbnorm{u}=1$ and $\cbnorm{u - \la\id_{\cc X}}\leq 1- \la$ for some $\la>0$. By Lemmas \ref{lem:paulsen-cp-map} and \ref{lem:cb-to-cp-index} we observe \[\la^{-1}\tilde u_1 - \id_{\widetilde{\cc X}} = \la^{-1}(\tilde u_1 - \la\id_{\widetilde{\cc X}})\in \CP_1(\widetilde{\cc X}).\]
    It can be easily checked that if $u(\cc X)\subset \cc X_0$ then $\tilde u_R(\widetilde{\cc X})\subset \widetilde{\cc X}_o$ for any $R\geq 0$. Thus $\la^{-1}\tilde u_{1}\in \La(\widetilde{\cc X},\widetilde{\cc X}_o)$ and $\Ind_{\cp}(\widetilde{\cc X}: \widetilde{\cc X}_o) \leq \|\la^{-1}\tilde u_{1}(1)\| = \la^{-1}$. \qedhere
\end{proof}

\begin{remark}
For any inclusion of operator spaces $\cc X_0\subset \cc X$, it follows that $\Ind_{\cp}(\widetilde{\cc X}: \widetilde{\cc X_0})\leq 2$ as witnessed by inducing the zero map $0: \cc X\to \cc X$ to $\widetilde{\cc X}$. Thus, the previous result cannot be used to show that the $\CB$-index can take arbitrarily large values. This is nonetheless true, as the following example shows.
\end{remark}

Any Banach space $X$ may be equipped with a \emph{minimal operator space structure} $\MIN(X)$, so that any for any linear map $T:X\to X$, we have that \[\|T\| = \|T:\MIN(X)\to\MIN(X)\|_{\cb}.\] Further, we have that for any closed subspace $X_0\subset X$, that $\MIN(X_0)\subset \MIN(X)$ completely isometrically. It follows that for any inclusion of Banach spaces $X_0\subset X$, we may define the \emph{bounded index} $\Ind_{\text{B}}(X:X_0)$ by
\begin{equation}
    \Ind_{\text{B}}(X:X_0) := \inf\{\|T\| : T\in\cc L(X,X_0),\ \|T - \id_X\|\leq \|T\|-1\}.
\end{equation}
As is the case of $\CB$-index we have that
\begin{equation*}
    \Ind_{\text{B}}(X:X_0)^{-1} = \sup\{\la\geq 0: T\in \cc L(X,X_0),\ \|T\|=1,\ \|T - \la\id_X\|\leq 1-\la\}.
\end{equation*}
By the previous remarks, it follows that
\[\Ind_{\CB}(\MIN(X): \MIN(X_0)) = \Ind_{\text{B}}(X: X_0).\]

\begin{example}
    Consider $\ell_\infty(n)$, and let $E: \ell_\infty(n)\to \ell_\infty(n)$ be the conditional expectation onto the constant sequences $\bb C 1_n$. That is, $E(x_1,\dotsc,x_n) = (\bar x, \dotsc, \bar x)$, where $\bar x  = \frac{1}{n} \sum_{i=1}^n x_i$. It is easy to see that $\|E: \ell_\infty(n)\to \ell_\infty(n)\|=1$. Notice that for $x = (1, \dotsc, 1, 0)$ we have that $\|E(x) - \la x\| = \frac{n-1}{n}$ for any $\la\geq 0$. Thus $\|E- \la\id\|\leq 1-\la$ only if $0\leq\la\leq 1/n$, and that any such value suffices. 
    
    Next we would like to show that for any $T:\ell_\infty(n)\to \bb C1_n$ with  $\|T\|=1$, we have that $\|T-\la\id\|\geq 1-\la$ for $\la\geq 1/n$. Let $T(e_i) = c_i 1_n$ for $c_i\in \bb C$, and note that $|c_1|+\dotsb + |c_n| = \|T\|=1$.  Without loss of generality we can assume that $|c_1|\geq |c_2| \geq \dotsb \geq |c_n|$, so that if $c_1,\dotsc, c_n$ is not constant, then $\|T(s) - \la s\|> \frac{n-1}{n}$ for $s = (s_1, \dotsc, s_{n-1}, 0)$, where $s_i = \bar c_i/|c_i|$ if defined and $s_i =0$, otherwise.
    Altogether, this shows that $\Ind_{\text{B}}(\ell_\infty(n) : \bb C1_n) = n$.
\end{example}

For a graph $\G = (V,E)$, pick an arbitrary vertex $v_0$, and consider the space ${\rm Lip_0}(\G)$ consisting of all functions $f: V\to \bb C$ with $f(v_0)=0$ under the discrete Lipschitz norm $\|f\|_{\rm Lip} = \max\{\|f(v) - f(w)\|: (v,w)\in E\}$. We may identify the dual ${\rm Lip}_0(\G)^*$ with the \emph{Lipschitz free space} $\cc F(\G)$ over $\G$: see \cite[Chapter 10]{Ostrovskii2013} for details. When equipped with the maximal operator space structure, this becomes the Lipschitz-free operator space $\cc F_{os}(\G)$ over $\G$ \cite{braga2021}.

\begin{question}
    For a graph $\G = (V,E)$ what are $\Ind(\cc F(\G) : \bb C1_V)$ and $\Ind_{\cb}(\cc F_{os}(\G) : \bb C1_V)$? Here $1_V(f) = \sum_{x\in V} f(x)$.
\end{question}

\section*{Acknowledgements}

Roy Araiza was partially supported as a J.L. Doob Research Assistant Professor at the University of Illinois at Urbana-Champaign. Griffin and Sinclair were partially supported by NSF grant DMS-2055155. The authors would like to thank Li Gao, and Sam Harris for their valuable input on an early draft of the manuscript.

\end{document}